\documentclass{article}
\usepackage{graphicx} 
\graphicspath{ {images/} }
\usepackage{booktabs,tabularx}
\usepackage{float}
\usepackage{caption}
\usepackage{indentfirst}
\usepackage{makecell}
\usepackage{changepage}
\usepackage{adjustbox}
\usepackage{multirow}
\usepackage{comment}
\usepackage{amsfonts}
\usepackage{verbatim} 
\usepackage{amsmath} 
\usepackage{amssymb} 
\usepackage{amsthm}
\usepackage{fullpage} 
\usepackage{paralist} 
\usepackage{listings}
\usepackage{enumitem} 
\usepackage{mathtools}
\usepackage[outline]{contour}
\usepackage{upgreek}
\usepackage{empheq}
\usepackage{caption}
\usepackage{mathrsfs}
\usepackage{derivative}
\usepackage[braket, qm]{qcircuit}
\usepackage{graphicx}
\usepackage[normalem]{ulem}
\usepackage{setspace}
\usepackage{hyperref}

\usepackage[sorting=none]{biblatex} 
\usepackage [english]{babel}
\usepackage [autostyle, english = american]{csquotes}
\MakeOuterQuote{"}

\newcommand{\floor}[1]{\left\lfloor #1 \right\rfloor}
\newcommand{\ceil}[1]{\left \lceil #1 \right \rceil}
\newcommand{\fractional}[1]{\left \{ #1 \right \}}

\newtheorem{theorem}{Theorem}[section]
\newtheorem{lemma}[theorem]{Lemma}

\title{Expression for $g(k)$ Related to Waring's Problem}
\author{Owen Root}

\begin{document}

\maketitle

\begin{abstract}
    Waring's Problem asks whether, for each positive integer $k$, there exists an integer $s$ such that every positive integer is a sum of at most $k$th powers. While Hilbert proved the existence of such $s$, Waring's Problem has lead to areas of related work, namely the function $g(k)$, which denotes the least such $s$. There is no known general closed form for $g(k)$, though for $g(k)$ has been evaluated for small $k$. Prior work has reduced the problem to verifying a particular condition, which if never occurs, implies an expression for $g(k)$. In this paper, I present a proof the condition never occurs, thus fixing the value of $g(k)$.
\end{abstract}

\section{Introduction}

Edward Waring asked in 1770 whether, for each positive integer $k$, there exists an integer $s$ such that every positive integer is a sum of at most $k$th powers \cite{Waring}. Hilbert answered this affirmatively in 1909 \cite{Hilbert}. Let 
$g(k)$ denote the least such s. Despite intense study since that time, there is no known general closed form for $g(k)$; only conditional formulas are known \cite{Dickson1936, KubinaWunderlich1990, Mahler1957, Niven1944, Pillai1936, Rubugunday1942}.
\par
A particularly compelling candidate is the explicit expression
\begin{equation}\label{eq:conj}
g(k)=2^k-\left\lfloor\left( \frac{3}{2}\right)^k\right\rfloor-2.
\end{equation}
which matches all known values of $g(k)$ for small $k$. Dickson, Pillai, Rubugunday, and Niven showed that the only possible failure of \eqref{eq:conj} is if there exits a $k$ such that
\begin{equation} \label{eq:carry}
2^k \left\{\left(\frac{3}{2}\right)^k\right\} +\left\lfloor \left( \frac{3}{2} \right)^k\right\rfloor > 2^k,
\end{equation}
where $\floor{x}$ and $\fractional{x}=x-\floor{x}$ denote the floor and fractional-part functions, respectively. If \eqref{eq:carry} never occurs, then \eqref{eq:conj} follows. The present paper proves that \eqref{eq:carry} does not occur for any positive integer $k$. Consequently, $g(k)$ is given exactly by \eqref{eq:conj}.

\section{Proof}

\subsection{Groundwork} \label{subsect: ground work}

To set up, let $f_1(k)=2^k$, the right side of \eqref{eq:carry}, and let $f_2(k)$ equal the left side of \eqref{eq:carry}, both of which can be seen in Fig. \ref{fig: steps} (note that in the figures, for graphical clarity, $k$ is treated as $\floor{x}$). $f_1$ and $f_2$ are "strict" in the sense that they are functions of only integer inputs. Consider the "strict" expression's "loose" forms (in the sense that they are functions of not just the integers), $F_1(x)=2^x$ and 
\begin{equation} \label{eq: looser one}
    F_2(x)=2^x \fractional{\left(\frac{3}{2}\right)^x} +\left\lfloor \left( \frac{3}{2} \right)^x\right\rfloor,
\end{equation}
where $x\in\mathbb{R}$ and $x>0$. As $\fractional{x}=x-\lfloor x\rfloor$, so we can further right \eqref{eq: looser one} as
\begin{equation} \label{eq: looser two}
    F_2(x)=2^x \left( \left( \frac{3}{2}\right)^x -\left\lfloor \left( \frac{3}{2} \right)^x \right\rfloor \right) +\left\lfloor \left( \frac{3}{2} \right)^x\right\rfloor.
\end{equation}
\par
As can be seen in Fig. \ref{fig: loose and strict forms}, $F_2(x)$ forms a "comb" of segments as the floor terms breaks the function into discontinuous pieces . Let $S_n$ denote the $n$th segment of the comb, where $n=1,2,3,...$ so the first segment is $S_1$. Each segment $S_n$ has $x$ values that form its upper and lower endpoints, denoted $L_n$ and $U_n$, respectively. $S_n\in[L_n,U_n)$, that is, all of the $L_n$ fall on \eqref{eq: looser two}'s lower envelope, 
\begin{equation} \label{eq: lower envelope}
    \phi_1(x)=(3/2)^x,
\end{equation}
while, due to the nature of the floor function, all of the $U_n$ asymptotically "fall" on it's upper envelope, 
\begin{equation} \label{eq: upper envelope}
    \phi_2(x)=2^x+(3/2)^x-1.
\end{equation} 
This behavior can be seen in Fig. \ref{fig: envelopes of S_n}. Each $S_n$ segment shares its lower bound with the upper bound of the previous segment, $L_n=U_{n-1}$, and shares its upper bound with the lower bound of the next segment, $U_n=L_{n+1}.$
\par
The location of the "steps" formed by $f_1(k)=2^k$ and $f_2(k)$ are located at wherever the "looser" functions, $F_1(x)=2^x$ and $F_2(x)$, cross integer values of $x$ (a.k.a the values of $k$). Hence we can recast the initial condition given, the inequality in \eqref{eq:carry}, to a simpler case: \emph{Are there any integer values of $x$ where  \eqref{eq: looser two}, $F_2(x)$, is greater than $F_1(x)=2^x$?} 
\par
In fact, we need not even consider the entirety of $F_2(x)$. Since \eqref{eq:carry} is a strict inequality, we only need to consider the portion that is greater than $F_1(x)=2^x$. This portion consists of the segments that form the top part of the "comb." Denote the $n$-th segment of the top part of the "comb," the portion $> 2^x$, as $s_n$. Each $s_n$ has $x$ values for its lower and upper endpoints, denoted $l_n$ and $u_n$, respectively. As with the larger segments, $s_n\in[l_n, u_n)$. The lower envelope of the $s_n$ segments is evidently $2^x$, while the upper envelope is the same as for the $S_n$ segments, $\phi_2$, so $U_n=u_n$. These envelopes can be seen in Fig. \ref{fig: envelopes of s_n}.
\par
As the $s_n$ segments are, by definition, greater than $2^x$, we can rephrase the question even more narrowly: \emph{Are there any $s_n$ segments that fall on an integer value of $x$?} That is, a segment that has a some midpoint at an integer $x$, hence an $l_n$ below the integer $x$ and an $u_n$ above the integer $x$.

\subsection{Values of Upper and Lower Bounds}
The limit of each $u_n$, or each $L_{n+1}$, is given by the solutions to 
\begin{equation}
    2^x \left( \left( \frac{3}{2}\right)^x -\left\lfloor \left( \frac{3}{2} \right)^x \right\rfloor \right) +\left\lfloor \left( \frac{3}{2} \right)^x\right\rfloor = \left(\frac{3}{2}\right)^x.
\end{equation}
Let $a=(3/2)^x$, then $2^x\{a\}+\floor{a}=a$. $\floor{a}$ is always an integer, $\floor{a}=n, \; n\in\mathbb{Z}$, and $\{a\}=a-n$. Hence (5) is recast as $2^x(a-n)=a-n$, or $2^x\{a\}=\{a\}$. This has solutions $\{a\}=0$ (so $a$ must be an integer) or $2^x=1$. The solution $x=0$ accounts for both, while all other solutions for $x$ are constrained by the requirement that $a$ be an integer. We now have $a=n\rightarrow \left( \frac{3}{2} \right)^x=n$. 
\par
Taking logarithms of $\left( \frac{3}{2} \right)^x=n$, we have that the limit of $u_n$, as approached from the left, is given by
\begin{equation} \label{eq: limit of u_n}
    \frac{\ln(n+1)}{\ln(3/2)}, \;\;n=1,2,3,...,
\end{equation}
where I have re-indexed $n$ so as to avoid $u_1=0$. The first several $u_n$ can be seen in Fig. \ref{fig:5}.
\par
Now let's turn our attention to $l_n$, which are given by the solutions to 

\begin{equation} \label{eq: l_n solution to}
     2^x \left( \left( \frac{3}{2}\right)^x -\left\lfloor \left( \frac{3}{2} \right)^x \right\rfloor \right) +\left\lfloor \left( \frac{3}{2} \right)^x\right\rfloor = 2^x.
\end{equation}

We proceed similarly to $u_n$; with $a=(3/2)^x$ and $\floor{a}=n$, we have $2^x(a-n)+n=2^x$. With some rearranging we have $2^x(n+1-(3/2)^x)=n$. Replacing the exponents with their base-$e$ equivalents, we get $e^{x\ln2}(n+1-e^{x\ln(3/2)})=n$. Now let $t=e^{x\ln(3/2)}=a$, so $x=\ln(t)/\ln(3/2)$. We can also relate $t$ with $e^{x\ln2}$, $t^q=e^{x\ln2}$, where $q=\ln(2)/\ln(3/2)$. \eqref{eq: l_n solution to} has now been recast to

\begin{equation}\label{eq: recast l_n solution to}
    t^q(n+1-t)=n \; \Rightarrow \; t(n)=\sqrt[q]{\frac{n}{n+1-t}},
\end{equation}
which is transcendental, with trivial solutions $t=1$. Consider the non-integer extension of \eqref{eq: recast l_n solution to}, $t^q(x+1-t)=x$, which is shown in Fig. \ref{fig:6} and can be denoted $t(x)$. The non-trivial solutions may be found numerically. 
\par
Since in \eqref{eq: recast l_n solution to} each integer $n$ gives a single non-trivial solution for $t(x=n)$, we can simply denote it $t_n$. To show this, consider the non-integer extension can be rearranged into
\begin{equation} \label{eq: less abig recast l_n solution to}
    x(t)=\frac{t^q-t^{q+1}}{1-t^q},
\end{equation}
which only contains the non-trivial solutions, as $t=1$ leaves it undefined. Note that we need only consider the regime $x>0$ and $t>0$. Consider now $\text{d}x/\text{d}t=x'(t)$,
\begin{equation} \label{eq: x'(t)}
    x'(t)=\frac{t^{q - 1} (t^{q + 1} - (q + 1) t + q)}{(t^{q} - 1)^{2}}.
\end{equation}
\par
Since $q=\ln(2)/\ln(3/2)>0$, for $t>0$ we have $t^{q-1}>0$ and $(t^q-1)^2>0.$ So to determine the sign of $x'(t)$ we then need to consider $t^{q+1}-(q+1)t+q$. Because $q>0$, we know $t^{q+1}$ is convex on $t>0$. The addition of the linear terms $-(q+1)t+q$ preserve the convexity. In fact $t^{q+1}-(q+1)t+q=0$ has the solution $t=1$, which, together with the convexity, ensures that $t^{q+1}-(q+1)t+q\ge0$ for $t>0$, with equality at $t=1$. Thus on the domain $t\in(0,1)\cup(1,\infty)$, we have $x'(t)>0$.
\par
In the limiting behavior, evidently $\lim_{t\rightarrow 0}x(t)=0$ and $\lim_{t\rightarrow \infty}x(t)=\infty$. At the discontinuity, $\lim_{t\rightarrow 1}x(t)=1/q$, by L'Hôpital's rule. Therefore, on the domain, $x\in(0,1/q)\cup(1/q,\infty), \;t\in(0,1)\cup(1,\infty)$, we have injectivity between $t$ and $x$. Since $1/q<1$, It follows that for each integer $n\ge1$ there exits a unique $t_n>1$ with $x(t_n)=n$. 
\par
By definition of $l_n$ and $u_n$, we know $u_{n-1}<l_{n}<u_{n}$, otherwise we would have segment with a lower endpoint that is not between the segment's upper end point and the upper end point of the previous segment, which is non-nonsensical. Canceling the $\ln(3/2)$ factors in each and exponentiating, this becomes $n<t_n<n+1$. Therefore, we can say that the solutions to \eqref{eq: l_n solution to} are given by 
\begin{equation} \label{eq: l_n values}
    l_n=\frac{\ln(t_n)}{\ln(3/2)},
\end{equation}
where $t_n$ is given by the solutions for $t$ in \eqref{eq: recast l_n solution to} for $n<t_n<n+1$, $n=1,2,3,...$. The first several $l_n$ are shown in Fig. \ref{fig:7}.

\subsection{Relationship Between the Bounds} \label{subsect: relat. bt. bounds}
With $u_n$ and $l_n$ established, we return to our earlier posed question: \emph{Are there any $u_n$ and $l_n$ that fall on opposite sides of an integer value of $x$?} That is, is there an $n>1$ such that that $\ceil{l_n}=\floor{u_n}$? Or equivalently, is there an $n$ where $\floor{l_n}=\floor{u_n}$ and $\ceil{l_n}=\ceil{u_n}$ are not true?  $\ceil{x}$ denotes the ceiling function (the smallest integer greater than or equal to $x$). 
\par
As established, $n<t_n<n+1$, which implies that $\ceil{t_n}=\ceil{n+1}=\floor{n+1}$, as $n+1$ is an integer. Consider then $\ln(n)<\ln(t_n)<\ln(n+1)$, but since $\ln$ maps integers to transcendentals, $\ceil{\ln(n+1)}\ne \floor{\ln(n+1)}$, and in fact $\ceil{\ln(n+1)}>\floor{\ln(n+1)}$. It remains the case though that $\ceil{\ln(t_n)}=\ceil{\ln(n+1)}$, given that $\ln(n)<\ln(n+a)<\ln(n+1)$, where $a\in \mathbb{R}, 0<a<1$ and $\ln(n+1)-\ln(n)<1$ for $n>1$.
\par
Consider now, 
\begin{equation} \label{eq: general ceil prop}
    \ceil{\frac{\ln(\alpha)}{b}}=\ceil{\frac{\ln(n+1)}{b}},
\end{equation}
where $b$ is some non-integer $>0$, and $\alpha$ is some non-integer such that $n<a<n+1$. Under what properties of $b$ does \eqref{eq: general ceil prop} hold? Since $n<\alpha<n+1$, $\ln(\alpha)<\ln(n+1)$, and so $\ln(\alpha)/b<\ln(n+1)/b$. 
\par
If $\ln(\alpha)/b>\floor{\ln(n+1)/b}$, then 
\begin{equation} \label{eq: step 1}
    \floor{\frac{\ln(n+1)}{b}}<\frac{\ln(\alpha)}{b}<\frac{\ln(n+1)}{b}.
\end{equation}
As $\ln(n+1)/b$ is not an integer,
\begin{equation} \label{eq: step 2}
    \floor{\frac{\ln(n+1)}{b}}<\frac{\ln(n+1)}{b}<\floor{\frac{\ln(n+1)}{b}}+1.
\end{equation}
Putting \eqref{eq: step 1} and \eqref{eq: step 2} together, we have
\begin{equation}
    \floor{\frac{\ln(n+1)}{b}}<\frac{\ln(\alpha)}{b}<\floor{\frac{\ln(n+1)}{b}}+1.
\end{equation}
Hence, $\ceil{\ln(\alpha)/b}=\floor{\ln(n+1)/b}+1.$ From \eqref{eq: step 2}, $\ceil{\ln(n+1)/b}=\floor{\ln(n+1)/b}+1$, and so 
\begin{equation}
    \ceil{\frac{\ln(\alpha)}{b}}=\ceil{\frac{\ln(n+1)}{b}}\ne\floor{\frac{\ln(n+1)}{b}}.
\end{equation}
Therefore, the conditions $\ln(\alpha)/b>\floor{\ln(n+1)/b}$ and $b$ not an integer, are sufficient conditions on $b$ for \eqref{eq: general ceil prop} to hold.
\par
Consider the case that $\alpha=t_n$ and $b:=\ln(3/2)$, so $\ln(\alpha)/b=l_n$. Here, $b$ is evidently not an integer, so the question is the condition that $l_n>\floor{\ln(n+1)/b}$ met? If so, then $\ceil{l_n}=\ceil{u_n}$. For convenience, define 
\begin{equation}
    z_n:=\floor{\frac{\ln(n+1)}{b}}.
\end{equation}
Notice that $l_n$ is strictly increasing with respect to $n$ and never an integer, while $z_n$ is integer-valued, non-decreasing, and changes only at those indices where $\ln(n+1)/b$ crosses an integer. Call any index $\ell$ with $z_{\ell+1}>z_\ell$ a \emph{jump index}.
\par
Since \(l_n\) is increasing and $z_n$ is constant on each block between jumps, the following monotonicity transfer holds: if for some jump index $\ell$ we have $l_\ell>z_\ell$, then $l_n>z_n$ for every $n$ in the entire block where $ z_n=z_\ell$. Equivalently, once the inequality $l_n>\lfloor \ln(n+1)/b\rfloor$ is true at the start of a block, it remains true throughout that block. This enables us to narrow the search to verify that $l_n>\floor{\ln(n+1)/b}$ to just the jump indexes $\ell$.
\par
We can specifically enumerate the $\ell$ integers. Let $i$ index the $\ell$ values (forgive the repeated indexing). The $l_i$ are exactly the integers whose unit interval $[n,n+1)$ contains a power of $3/2$, equivalently, the $\ell_i$ values are given by 
\begin{equation} \label{eq: jump indexes}
    \ell_i=\floor{\left(\frac{3}{2}\right)^i}, 
\end{equation}
the first few of which are $l_i=1,2,3,5,7,11,17,\dots$, though since we are only concerned with $n\ge 2$, we also only need to be concerned with $i\ge2$.

\subsection{Some Useful Lemmas}
Now consider the following lemmas, which will become useful shortly. 
\begin{lemma} \label{lem: 1}
    For $n$ integer and $n>1$, with $b=\ln(3/2)$,
        \begin{equation} \label{eq: lemma 1}
        \floor{\frac{\ln(\floor{\left(\frac{3}{2}\right)^n}+1)}{b}}=n.
    \end{equation}
\end{lemma}
\begin{proof}
Showing that \eqref{eq: lemma 1} is equivalent to showing that $n<\log_{3/2}(\floor{(3/2)^n}+1)<n+1$, which then operation of the floor function makes it equal to $n$.
\par
First consider the lower bound. For $n\ge1$, $(3/2)^n$ is not an integer, hence $\lfloor (3/2)^n\rfloor+1>(3/2)^n$. The base $3/2>1$ makes $\log_{3/2}$ strictly increasing, so
\begin{equation}
    \log_{3/2}\left(\floor{\left(\frac{3}{2}\right)^n}\right)>\log_{3/2}\left(\left(\frac{3}{2}\right)^n\right)=n.
\end{equation}
\par
Now consider the upper bound. If $n>1$ then $(3/2)^n\ge (3/2)^2=2.25$. Since $(3/2)^n$ is not an integer, $\floor{(3/2)^n}+1<(3/2)^n+1$. Because $(3/2)^n> 2$, we have $(3/2)^n+1< (3/2)^n+(3/2)^n/2=(3/2)^{n+1}$. Hence $\floor{(3/2)^n}+1<(3/2)^{n+1}$. The monotonicity of $\log_{3/2}$ gives
\begin{equation}
   \log_{3/2}\left(\floor{\left(\frac{3}{2}\right)^n}+1\right)<\log_{3/2}\left(\left(\frac{3}{2}\right)^{n+1}\right)=n+1. 
\end{equation}
\par
Combining the two bounds we have
$n<\log_{3/2}(\floor{(3/2)^n}+1)<n+1$, thus via operation by the floor function we arrive at \eqref{eq: lemma 1}.
\end{proof}

\begin{lemma}\label{lem: 2}
    Let $t_{f(x)}$ be given by the solutions to $t_{f(x)}^q(f(x)+1-t_{f(x)})=f(x)$, where $q>1$, $x>0$, $f(x)$ is some injective function of $x$ such that $f(x)\ge1$. Then 
    \begin{equation} \label{eq: lemma 2}
        f(x)<t_{f(x)}<f(x)+1.
    \end{equation}
\end{lemma}
\begin{proof}
    This is essentially a generalization of what we have already seen about $t(x)$ and $t_n$. As with those special cases discussed earlier, on the domain, $x\in(0,1/q)\cup(1/q,\infty), \;t\in(0,1)\cup(1,\infty)$, we have injectivity between $x, f(x)$, and $ t_{f(x)}$. Define 
    \begin{equation}
        h_{f(x)}(t):=t^q(f(x)+1-t).
    \end{equation}
    \par
    Since $f(x)\ge1$, 
    \begin{equation}
        h_{f(x)}(f(x))=f(x)^q(f(x)+1-f(x))=f(x)^q>f(x)
    \end{equation}
    and 
    \begin{equation}
        h_{f(x)}(f(x)+1)=(f(x)+1)^q(f(x)+1-(f(x)+1))=0<f(x).
    \end{equation}
    Then, by continuity with respect to $t$ and the injectivity, there is a unique solution $t\in(f(x), f(x)+1)$, which is equivalent to \eqref{eq: lemma 2}.
    
\end{proof}

\begin{lemma} \label{lem: 3}
    Let $t_{f(x)}$, $q$, $x$, and $f(x)$ have the same properties as in Lemma \ref{lem: 2}. Let $\ell(x)$ and $g(x)$ be functions with $l(f(x))>0$ for the $x$ under consideration. 
    \par
    Assume that $g$ is strictly increasing on $x\in(0,\infty)$ \emph{(equivalently, it suffices that $g$ is strictly increasing on the interval}
    \[
        \left[\ell(f(x))+1-\ell\left(f(x)\right)^{1-q},\ell(f(x))+1\right].
    \]
    If, in addition, 
    \begin{equation}\label{eq: endpoints}
        g\left(\ell(f(x))+1-\ell\left(f(x)\right)^{1-q}\right)>f(x)
        \quad\text{and}\quad
        g\left(\ell(f(x))+1\right)<f(x)+1
    \end{equation}
    hold, then
    \begin{equation}
        f(x)<g\left(t_{\ell(f(x))}\right)<f(x)+1.
    \end{equation}
\end{lemma}

\begin{proof}
    Write $u:=f(x)$ and $y:=l(u)$. Since $q>1$ and $y>0$, the unique solution $t_y\in(y,y+1)$ of $t^{q}\,(y+1-t)=y$ satisfies the \emph{basic bound}
    \begin{equation}\label{eq: basic-bound}
        y+1-y^{1-q}\le t_y < y+1.
    \end{equation}
    This can be seen by considering $t_y=y+\varepsilon$ with $0<\varepsilon<1$. Substituting back in, we have
    $(y+\varepsilon)^{q}(1-\varepsilon)=y$, hence
    \[
        1-\varepsilon=\frac{y}{t_y^{q}}\le \frac{y}{y^{q}}=y^{1-q},
    \]
    which is equivalent to \eqref{eq: basic-bound}.

    Since $g$ is strictly increasing on $x\in(0,\infty)$, it is in particular strictly increasing on 
    $\left[y+1-y^{1-q},\,y+1\right]$. Applying $g$ to \eqref{eq: basic-bound} gives
    \[
        g\left(y+1-y^{1-q}\right) < g(t_y) < g(y+1).
    \]
    Then, as long as the inequalities in \eqref{eq: endpoints} hold, with $y=\ell(u)$, we obtain 
    $u<g\!\big(t_{\,\ell(u)}\big)<u+1$, that is,
    \[
        f(x)<g\left(t_{\ell(f(x))}\right)<f(x)+1.
    \]
\end{proof}

\subsection{Endpoint Inequalities}
Let us make use of some of the established facts. We've established that if
\begin{equation} \label{eq: necessary condition for ceil equality}
    l_n>\floor{\frac{\ln(n+1)}{b}} \; \Rightarrow \; \frac{\ln(t_n)}{b}>\floor{\frac{\ln(n+1)}{b}},
\end{equation}
($b=\ln(3/2)$) then $\ceil{l_n}=\ceil{u_n}$, and furthermore that if the inequality holds for the jump indexes $n=\ell$, then it holds for all $n\ge2$. Using \eqref{eq: jump indexes} and Lemma \ref{lem: 1}, we have that 
\begin{equation} \label{eq: get i from l_i}
    \floor{\frac{\ln(\ell_i+1)}{b}}=i.
\end{equation}
\par
Let $f(x)=\floor{x}=i$ and $\ell(x)=\floor{(3/2)^x}$, so $\ell(f(x))$ is equivalent to \eqref{eq: jump indexes}. Let $g(x)=\ln(x)/b$. From Lemma \ref{lem: 2}, we have 
\begin{equation}
    \ell(f(x))<t_{\ell(f(x))}<\ell(f(x))+1 \;\Longleftrightarrow\; \ell_i<t_{\ell_i}<\ell_i+1.
\end{equation}
Then from Lemma \ref{lem: 3}, we have
\begin{equation} \label{eq: g(x) equivelences}
    f(x)<g(t_{\ell(f(x))})<f(x)+1 \Longleftrightarrow i<\frac{\ln(t_{\ell_i})}{b}<i+1,
\end{equation}
\emph{if the endpoint inequalities \eqref{eq: endpoints} hold here.} Combining \eqref{eq: get i from l_i} and \eqref{eq: g(x) equivelences}, we have 
\begin{equation} \label{eq: jump index inequality}
    i=\floor{\frac{\ln(\ell_i+1)}{b}}<\frac{\ln(t_{\ell_i})}{b},
\end{equation}
which is exactly the condition \eqref{eq: necessary condition for ceil equality}. Thus to show that the condition is true, it suffices to show that the endpoint inequalities \eqref{eq: endpoints} are true for this $f(x), \ell(x)$, and $g(x)$.
\par 
First, we consider 
\begin{equation} \label{eq: upper inequality starting}
    g\left(\ell(f(x))+1\right)<f(x)+1 \; \Longleftrightarrow \; \frac{\ln\left(\floor{(3/2)^i}+1\right)}{b}<i+1.
\end{equation}
Since $x\mapsto (3/2)^x$ and $\ln$ are increasing and $(3/2)^i\notin\mathbb{Z}$ for $i\ge 1$, we have
\begin{equation} \label{eq: upper inequality step 1}
    \frac{\ln\left(\floor{(3/2)^i}+1\right)}{b}
    \;<\;
    \frac{\ln\left((3/2)^i+1\right)}{b}
    \;\le\;
    \frac{\ln\left((3/2)^x+1\right)}{b}.
\end{equation}

Consider the functions
\[
F(x)=\frac{\ln\!\big((3/2)^x+1\big)}{b}
\qquad\text{and}\qquad
G(x)=x+1.
\]
Then
\[
F'(x)=\frac{(3/2)^x}{(3/2)^x+1}\in(0,1)
\quad\text{and}\quad
G'(x)=1.
\]
Hence $H(x):=G(x)-F(x)$ satisfies
\begin{equation}
    H'(x)=1-\frac{(3/2)^x}{(3/2)^x+1}=\frac{1}{(3/2)^x+1}>0,
\end{equation}
so $H$ is strictly increasing on $(0,\infty)$. $F(x)$ and $G(x)$ have a single equality at $x=-\frac{\ln\!\big((3/2)-1\big)}{b}\approx 1.7095$
Therefore $H(x)>0$ for all $x\ge 2$, that is,
\begin{equation} \label{eq: upper inequality step 2}
    \frac{\ln((3/2)^x+1)}{b}<x+1
    \qquad (x\ge 2).
\end{equation}

Applying \eqref{eq: upper inequality step 2} with $x\rightarrow i=\floor{x}$ and combining with \eqref{eq: upper inequality step 1} yields
\begin{equation}
\frac{\ln\left(\floor{(3/2)^{i}}+1\right)}{b} < \frac{\ln\left((3/2)^{i}+1\right)}{b} < i+1,
\end{equation}
thus \eqref{eq: upper inequality starting} holds.
\par
The other endpoint inequality, 
\begin{equation} \label{eq: lower endpoint start}
    g\left(\ell(f(x))+1-\ell\left(f(x)\right)^{1-q}\right)>f(x) 
    \;\Longleftrightarrow\;
    \frac{\ln\left(\floor{(3/2)^i}+1-\floor{(3/2)^i}^{1-q}\right)}{b}>i
\end{equation}
is considerably trickier and deserves a dedicated lemma.  
\par
\begin{lemma}\label{lem: lower-endpoint}
Let $b:=\ln(3/2)$ and $q:=\ln 2\,/\,b\in(1,2)$. For $x\ge 2$ set $i:=\floor{x}$ and $\ell_i:=\floor{(3/2)^i}$. Then
\begin{equation}\label{eq: lower-endpoint-claim}
    \frac{\ln\!\big(\ell_i+1-\ell_i^{\,1-q}\big)}{b}\;>\;i,
\end{equation}
equivalently,
\begin{equation}\label{eq: lower-endpoint-claim-mult}
    \ell_i+1-\ell_i^{\,1-q}\;>\;\Big(\frac{3}{2}\Big)^{\!i}.
\end{equation}
\end{lemma}

\begin{proof}
Let $m:=\ell_i=\floor{(3/2)^i}$ and $r:=\{(3/2)^i\}\in(0,1)$ so that $(3/2)^i=m+r$. Let $R_i:=3^i-2^i m=2^i r\in\{1,\dots,2^i-1\}$. Since $\ln$ is strictly increasing and $b>0$, \eqref{eq: lower-endpoint-claim} is equivalent to \eqref{eq: lower-endpoint-claim-mult}.

Introduce the function used earlier in Lemma \ref{lem: 2},
\[
h_m(t):=t^{q}\,(m+1-t),\qquad t\in(m,m+1).
\]
Then
\[
h_m'(t)=t^{q-1}\big(q(m+1-t)-t\big).
\]
For $t\in[m,m+1]$ we have $q(m+1-t)\le q<m\le t$ (since $q<2$ and $m\ge 2$ for $i\ge 2$), hence $h_m'(t)<0$ on $[m,m+1]$ and $h_m$ is strictly decreasing there.

Because $h_m$ is strictly decreasing, \eqref{eq: lower-endpoint-claim-mult} is \emph{equivalent} to
\begin{equation}\label{eq: hm-comparison}
    h_m\!\Big(\Big(\tfrac{3}{2}\Big)^i\Big)\;>\;h_m\!\big(m+1-m^{\,1-q}\big).
\end{equation}
We now bound the two sides of \eqref{eq: hm-comparison} in opposite directions.

\medskip
\noindent\emph{Upper bound for $h_m(m+1-m^{\,1-q})$.}
Let $\delta:=m^{\,1-q}\in(0,1)$. For $q\in[1,2]$ and $a,b\ge 0$ one has
\begin{equation}\label{eq: binomial-q}
(a+b)^q\le a^q+q\,a^{q-1}b+b^q.
\end{equation}
Which can be seen via the following. By homogeneity, set $t=b/a$ and define
$h_q(t):=1+q t+t^q-(1+t)^q$, so $h_q(0)=0$ and
$h_q'(t)=q\big(1+t^{q-1}-(1+t)^{q-1}\big)\ge 0$ for $q\in[1,2]$, hence $h_q(t)\ge 0$ for $t\ge 0$. Applying \eqref{eq: binomial-q} with $a=m$ and $b=1-\delta$, then multiplying by $\delta$ and using $\delta=m^{1-q}$, $m^{q-1}\delta^2=\delta$, and $(1-\delta)^q\le 1$, yields
\begin{equation}\label{eq: rhs-upper}
h_m(m+1-\delta)
=(m+1-\delta)^q\,\delta
\le m+q+\delta\big((1-\delta)^q-q\big)\;<\;m+q.
\end{equation}

\medskip
\noindent\emph{Lower bound for $h_m\big((3/2)^i\big)$.}
Using $(3/2)^{iq}=2^i$ and $m+1-(3/2)^i=1-r$, we have
\begin{equation}\label{eq: lhs-exact}
h_m\!\Big(\Big(\tfrac{3}{2}\Big)^i\Big)=2^i(1-r)=2^i-R_i.
\end{equation}
We claim that for $i\ge 3$,
\begin{equation}\label{eq: mR-bound}
m+R_i\le 2^i-2.
\end{equation}
Indeed, $m+R_i\neq 2^i$ since $R_i\in\{1,\dots,2^i-1\}$. If $m+R_i=2^i-1$ then from $3^i=2^i m+R_i$ we get
\[
3^i=(2^i-1)(m+1).
\]
If $i$ is even, unique factorization forces $2^i-1=3^s$, which is false for even $i\ge 4$. If $i$ is odd, then $2^i-1\equiv 1\pmod 3$, so $3^i\mid (m+1)$ and hence $m+1\ge 3^i$, contradicting $m+1=\lceil(3/2)^i\rceil<3^i$. Thus \eqref{eq: mR-bound} holds, and with \eqref{eq: lhs-exact} we obtain
\begin{equation}\label{eq: lhs-lower}
h_m\!\Big(\Big(\tfrac{3}{2}\Big)^i\Big)=2^i-R_i\;\ge\;m+2 \qquad (i\ge 3).
\end{equation}

\medskip
\noindent\emph{Comparison for $i\ge 3$.}
Since $q<2$, we have $m+q<m+2$. Combining \eqref{eq: rhs-upper} and \eqref{eq: lhs-lower} yields, for all $i\ge 3$,
\[
h_m\!\Big(\Big(\tfrac{3}{2}\Big)^i\Big)\;\ge\;m+2\;>\;m+q\;>\;h_m\!\big(m+1-m^{\,1-q}\big),
\]
which is exactly \eqref{eq: hm-comparison}.

\medskip
\noindent\emph{Base cases.}
For $i=3$, $m=\floor{(3/2)^3}=3$ and $R_3=3$, hence by \eqref{eq: lhs-exact}
\[
h_m\!\Big(\Big(\tfrac{3}{2}\Big)^3\Big)=8-3=5.
\]
On the other hand, \eqref{eq: rhs-upper} gives $h_m(m+1-m^{1-q})<m+q=3+q<5$ since $q<2$. Thus \eqref{eq: hm-comparison} holds.

For $i=2$, the inequality \eqref{eq: lower-endpoint-claim-mult} reads
\[
2+1-2^{\,1-q}>\Big(\tfrac{3}{2}\Big)^2
\quad\Longleftrightarrow\quad
\frac{3}{4}>2^{\,1-q}.
\]
Taking logarithms and noting that $\ln(3/4)<0$ gives
\[
\ln\!\Big(\tfrac{3}{4}\Big)>(1-q)\ln 2
\quad\Longleftrightarrow\quad
q>1+\frac{\ln(4/3)}{\ln 2}\approx 1.415\ldots
\]
which is satisfied by our fixed $q=\ln 2/\ln(3/2)\approx 1.7095$. Hence \eqref{eq: hm-comparison} also holds for $i=2$.

\medskip
In all cases $i\ge 2$, we have established \eqref{eq: hm-comparison}. By strict decrease of $h_m$ on $[m,m+1]$, this is equivalent to \eqref{eq: lower-endpoint-claim-mult}, which is equivalent to \eqref{eq: lower-endpoint-claim}.
\end{proof}
\par
Thus given our $f(x), \ell(x)$, and $g(x)$, via Lemma \ref{lem: lower-endpoint}, \eqref{eq: lower endpoint start} also holds. 

\subsection{End Result}
Since both the endpoint inequalities, \eqref{eq: upper inequality starting} and \eqref{eq: lower endpoint start}, hold for $i\ge2$, the inequality for the jump indexes, \eqref{eq: jump index inequality}, then holds. As discussed at the end of Section \ref{subsect: relat. bt. bounds}, since the inequality holds for the jump indexes, it also holds for all $n\ge 2$, that is
\begin{equation}
    \frac{\ln(t_n)}{b}>\floor{\frac{\ln(n+1)}{b}}.
\end{equation}
As discussed at the start of Section \ref{subsect: relat. bt. bounds}, this is a sufficient condition for the following to be true,
\begin{equation}
\ceil{\frac{\ln(t_n)}{b}}=\ceil{\frac{\ln(n+1)}{b}}
\;\Longleftrightarrow\;
\ceil{l_n}=\ceil{u_n}.
\end{equation}
\par
Since $u_n$ is never an integer for $n\ge2$, $\ceil{u_n}\ne\floor{u_n}$. And so we have that $\ceil{l_n}\ne\floor{u_n}$. Thus, for $n\ge2$, \emph{$l_n$ and $u_n$ never fall on opposite sides of an integer value of $x$}, and there are no such $s_n$ segments that have a midpoint at an integer $x$. Finally then, via the work of Section \ref{subsect: ground work}, no such $k$ exists that can make \eqref{eq:carry} true, and so the expression for $g(k)$ is given precisely by \eqref{eq:conj}.

\begin{figure}
    \centering
    \includegraphics[width=0.6\linewidth]{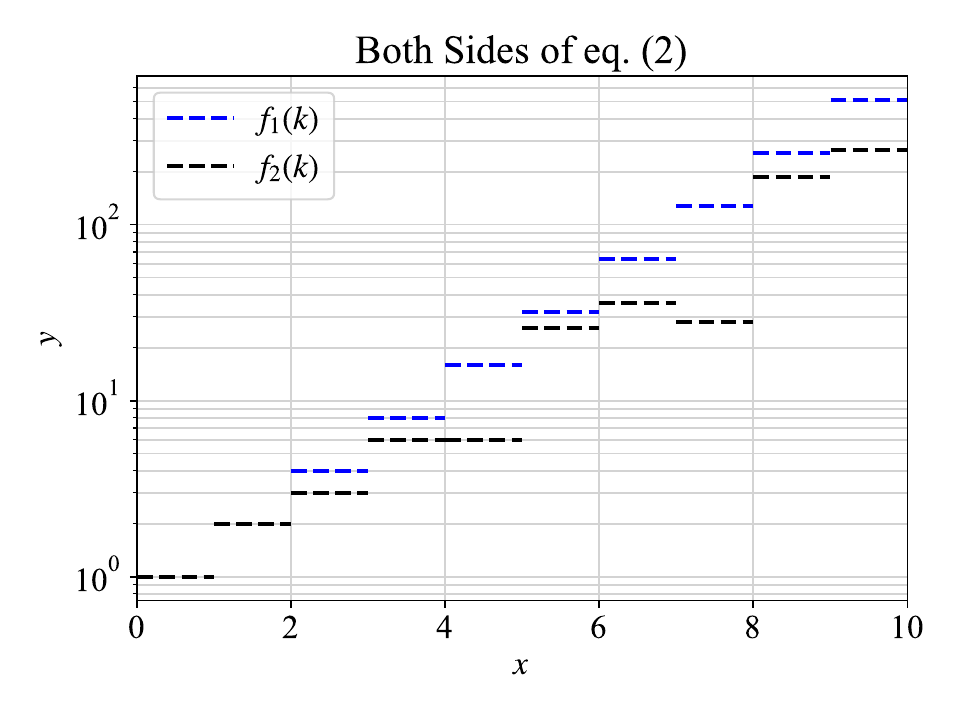}
    \caption{$f_1(k)$ and $f_2(k)$. For graphical clarity, $k$ has been treated as $\floor{x}$. Note that $f_1=f_2$ at $k=0$ and $k=1$.}
    \label{fig: steps}
\end{figure}

\begin{figure}
    \centering
    \includegraphics[width=0.6\linewidth]{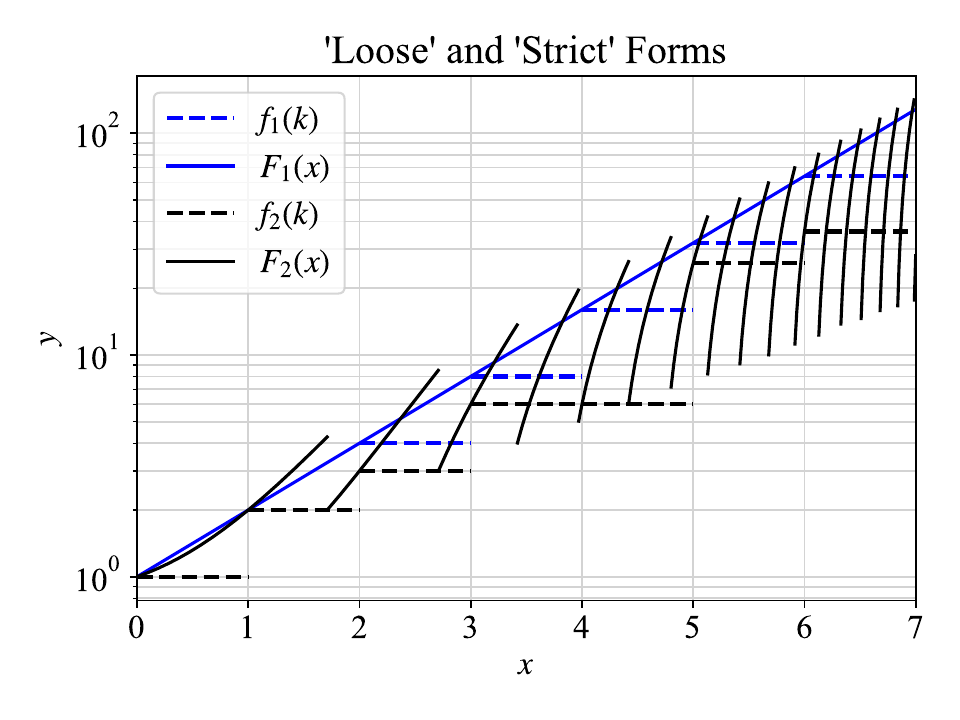}
    \caption{The "strict" expressions as dashed lines and the "loose" expressions as solid ones. $f_1(k)$ has it's "steps" starting at every location where $F_1(x)$ crosses an integer value of $x$, which similarly is true between $f_2(k)$ and $F_2(x)$.}
    \label{fig: loose and strict forms}
\end{figure}

\begin{figure}
    \centering
    \includegraphics[width=0.6\linewidth]{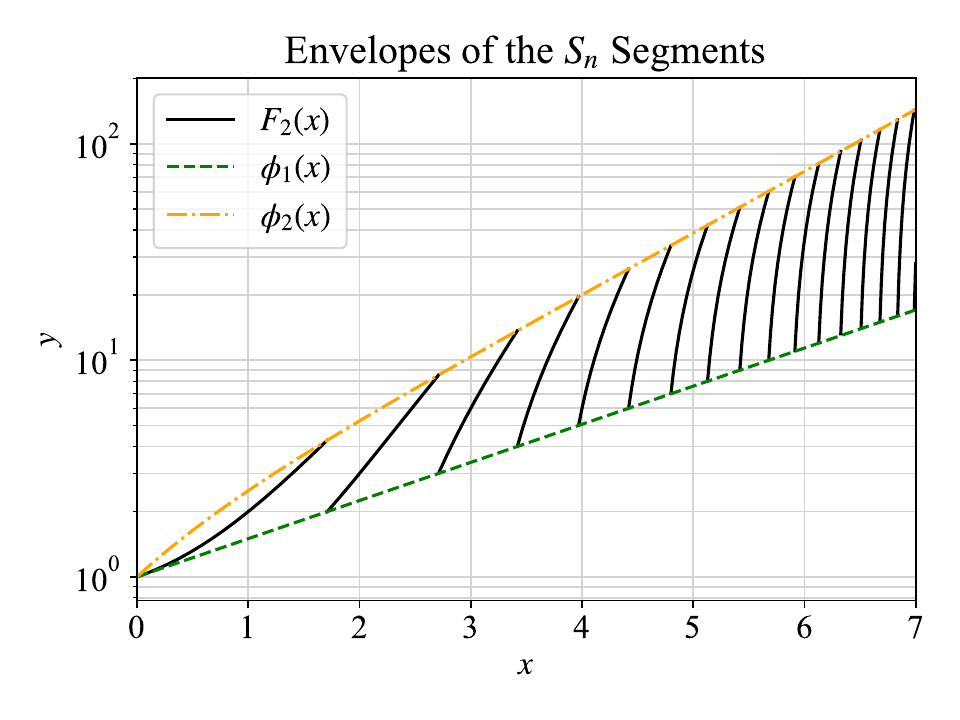}
    \caption{The upper and lower envelopes, $\phi_2(x)$ and $\phi_1(x)$, respectively, of $F_2(x)$ and hence the envelopes of the $S_n$ segments.}
    \label{fig: envelopes of S_n}
\end{figure}

\begin{figure}
    \centering
    \includegraphics[width=0.6\linewidth]{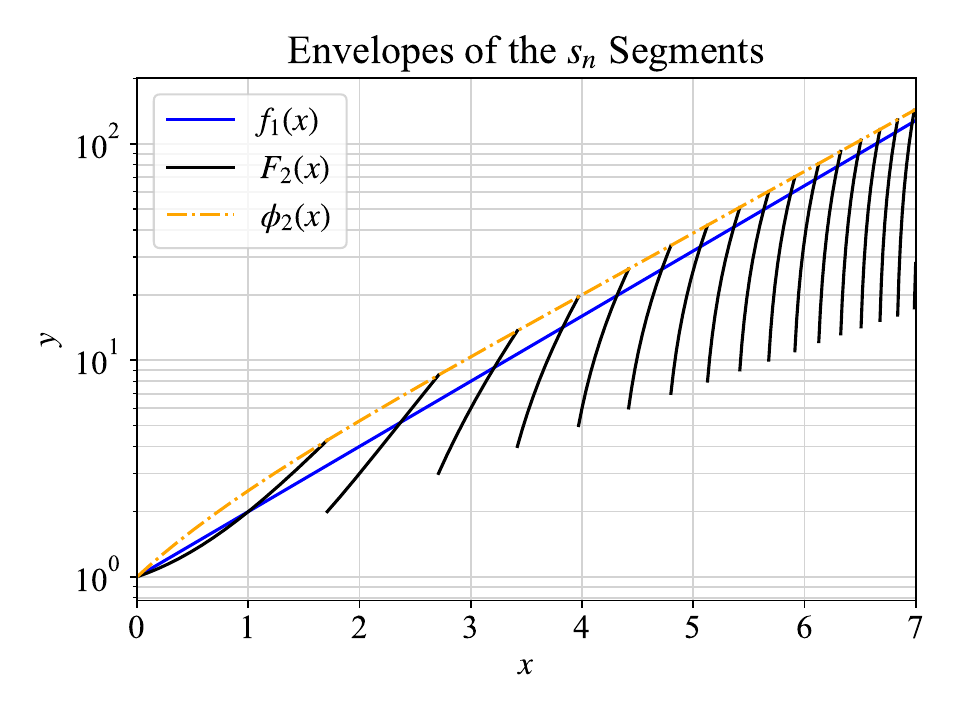}
    \caption{The upper and lower envelopes, $\phi_2(x)$ and $f_1(x)$, respectively, of the $s_n$ segments.}
    \label{fig: envelopes of s_n}
\end{figure}

\begin{figure}
    \centering
    \includegraphics[width=0.6\linewidth]{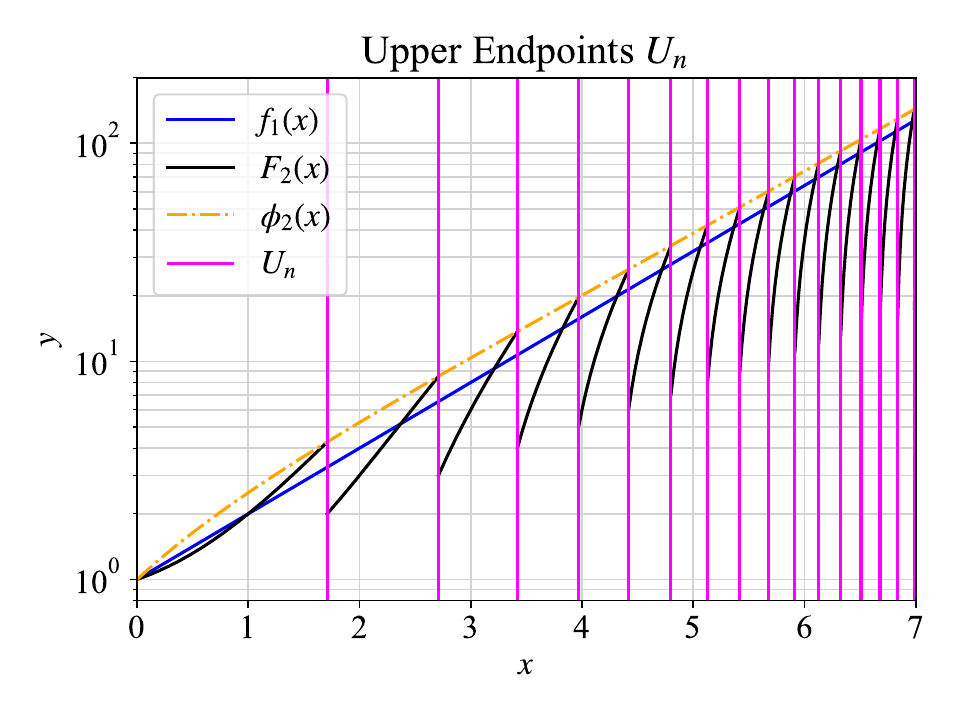}
    \caption{The limits of the upper boundary of the segments, $U_n=u_n$, which are equal to the lower bound of the next segment $L_{n+1}$, as given by $\ln{(n)}/\ln{(3/2)}$.}
    \label{fig:5}
\end{figure}

\begin{figure}
    \centering
    \includegraphics[width=0.6\linewidth]{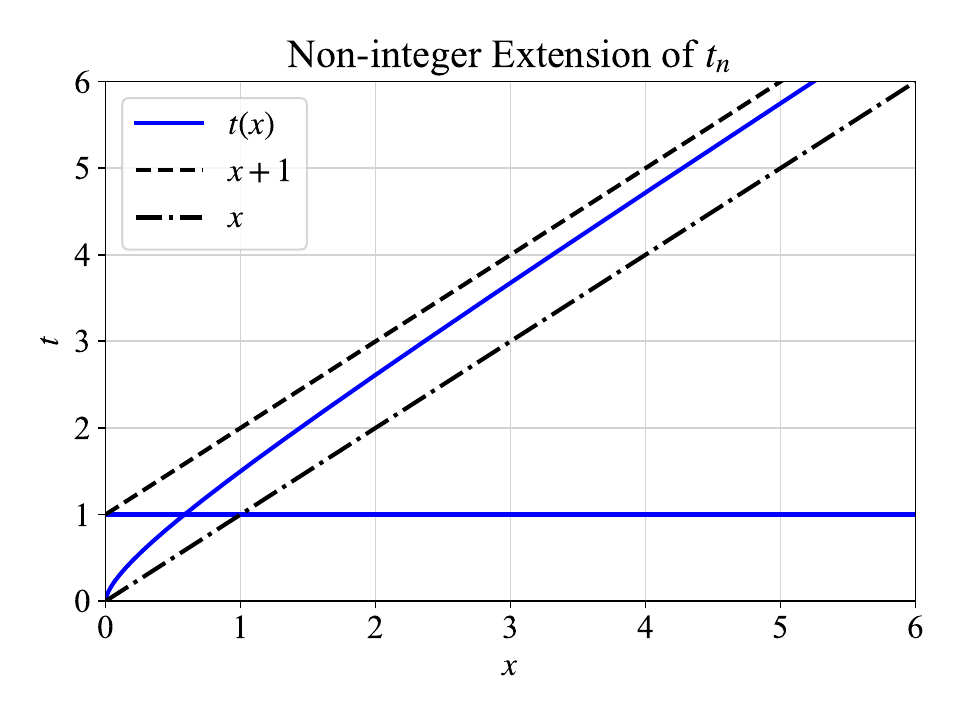}
    \caption{The non-integer extension of of $t_n$, $t^q(x+1-t)=x$. The trivial solutions form the horizontal line at $t=1$, while the non-trivial solutions are bounded by $x$ from below and $x+1$ from above.}
    \label{fig:6}
\end{figure}

\begin{figure}
    \centering
    \includegraphics[width=0.6\linewidth]{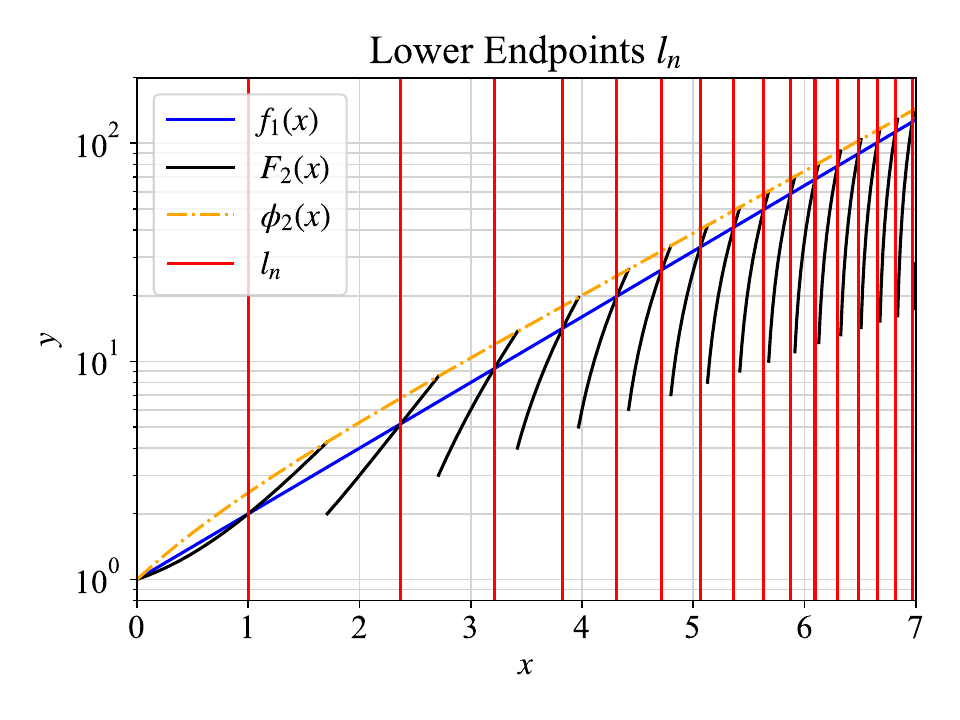}
    \caption{The lower bounds of the $s_n$ segments, $l_n$. }
    \label{fig:7}
\end{figure}

\newpage

\printbibliography

\end{document}